\documentclass{www-notes}
\title{Dispersive estimates for Klein-Gordon equations via a physical space approach}
\author{Willie Wong\Affiliation{Michigan State University, East Lansing, USA; \url{wongwwy@math.msu.edu}}}

\newcommand\dvol{\mathrm{dvol}}

\begin{document}
\maketitle

\begin{wwwabstract}
	Building on the hyperboloidal foliation approach of Lefloch and Ma, we extend Klainerman's physical-space approach to dispersive estimates to recover the frequency-restricted $L^1$--$L^\infty$ dispersive estimates for Klein-Gordon equations. The hyperboloidal foliation approach naturally only provide estimates within a fixed forward light-cone, and is based on an initial data norm that is not translation invariant. Both of these problems can be handled with frequency-dependent physical-space truncations. To handle the lack of scaling symmetry for the Klein-Gordon equation and complete the argument, we also need to keep track of the effectiveness of our estimates in the vanishing mass limit. 
\end{wwwabstract}

\section{Introduction}

Fundamental to the study of dispersive equations are the \emph{dispersive estimates} they enjoy, such estimates (for e.g.\ the Airy, Schr\"odinger, and wave equations) typically deriving from oscillatory integration control of the explicit Fourier representations of the corresponding fundamental solutions (see \cite{Tao2006} and references therein). 
An alternative physical-space method relying weighted energy estimates was first introduced by Klainerman \cite{Klaine2001} for the wave equation in dimension $d \geq 3$, and extended to the Schr\"odinger equation by the author \cite{Wong2018}. 
In this short note we present a physical-space method for studying the dispersive estimates enjoyed by the Klein-Gordon equation. 
Unlike the other model dispersive equations (Airy, Schr\"odinger, and wave), the Klein-Gordon equation does not enjoy scaling homogeneity. 
Its \emph{low frequency} components are known to behave more akin to solutions to the Schr\"odinger equation, while its \emph{high frequency} components behave more like waves, which complicates the study of its dispersive estimates; see \cite{NakSch2011} for a detailed discussion using the oscillatory integration method. 
Therefore it seems worthwhile to indicate how the same frequency-restricted results can be derived using a physical-space argument. 

Our approach will be a combination of Klainerman's physical space approach \cite{Klaine2001} for studying dispersive estimates of the wave equation, which is based on weighted energy estimates and the Klainerman-Sobolev inequalities adapted to constant time hypersurfaces, with global-Sobolev inequalities and energy estimates adapted to hyperboloidal foliations introduced in \cite{LefMa2015, Klaine1985a}. 
The weighted energy estimates used in the hyperboloidal foliation method are spatially inhomogeneous, and we will paste together frequency-adjusted cut-offs to localize the dependence on initial data.

\subsection{Acknowledgements} WWY Wong is supported by a Collaboration Grant from the Simons Foundation, \#585199.

\section{Preliminaries}

For convenience, all functions are assumed to be smooth. 

Throughout we will denote by $\Box_m$, for $m \geq 0$, the differential operator
\begin{equation}
	\Box_m \eqdef -\partial^2_{tt} + \underbrace{\sum_{i = 1}^d \partial^2_{ii}}_{\triangle} - m^2 
\end{equation}
where as indicate we will be working on the space-time $\Real^{1,d} \cong \Real \times\Real^d\ni (t,x)$. 
The homogeneous Klein-Gordon equation with mass parameter $m$ is the equation
\begin{equation}\label{eq:kg-m}
	\Box_m \phi = 0.
\end{equation}
For convenience we will typically prescribe our initial data at $t = 2$:
\begin{subequations}\label{eq:kg-id}
	\begin{align}
		\phi(2,x) &= f(x), \\
		\partial_t(2,x) &= g(x).
	\end{align}
\end{subequations}
In the hyperboloidal method (see \cite{LefMa2015}; also \cite{Wong2017p}) we consider the level-sets $\Sigma_\tau$ of the function
\begin{equation}
	\tau \eqdef \sqrt{t^2 - |x|^2}
\end{equation}
defined on the subset $\set{ (t,x) \in\Real^{1,d}}{t > \abs{x}}$. These level sets are hyperboloids asymptotic to the cone $\{t = |x|\}$. The Minkowski metric on $\Real^{1,d}$ induces Riemannian metrics on $\Sigma_\tau$, and we write $\dvol_\tau$ for the corresponding induced volume form. 
We also define the $\Sigma_\tau$-tangential vector fields $L^i$, where $i$ ranges from $1, \ldots, d$
\begin{equation}
	L^i \eqdef x^i \partial_t + t \partial_{x^i}.
\end{equation}
The fundamental energy identity for the Klein-Gordon equation, adapted to $\Sigma_\tau$ hypersurfaces, read \cite{LefMa2015, Wong2017p}
\begin{lem}\label{lem:energid}
	Let $\phi$ solve \eqref{eq:kg-m} with initial data \eqref{eq:kg-id}, then for every $\tau > 0$ the following weighted energy inequality holds, provided $f\in H^1$ and $g\in L^2$:
	\[
		\mathscr{E}_m(\phi,\tau) \eqdef \int_{\Sigma_\tau} \frac{1}{t\tau} \sum_{i = 1}^d (L^i\phi)^2 + \frac{\tau}{t} (\partial_t\phi)^2 + \frac{t}{\tau} m^2 \phi^2 ~\dvol_{\tau}
		\leq \int_{\Real^d} g^2 + |\nabla f|^2 + m^2 f^2 ~\D{x}.
	\]
	Furthermore, if the supports $\supp(f), \supp(g) \Subset B(0,2)$, then the equality is achieved. 
\end{lem}

The basic point-wise estimate is the following global Sobolev inequality \cite{Wong2017p}:
\begin{prop}\label{prop:globsob}
	Given any $\ell\in \Real$, there exists a constant $C$ depending only on the dimension $d$ and the number $\ell$, such that for any $\tau > 0$ 
	\[ \norm[L^\infty(\Sigma_\tau)]{ \tau^{1-\ell} t^{d + \ell - 1} \psi^2} \leq C \sum_{k= 0}^{\lfloor \frac{d}2\rfloor + 1} \sum_{i_1, \ldots ,i_k = 1}^d \int_{\Sigma\tau} \frac{t^\ell}{\tau^\ell} \abs{L^{i_1} L^{i_2}\cdots L^{i_k} \psi}^2 ~\dvol_\tau.\]
\end{prop}
Note that the constant in the above proposition is independent of $\tau$.
The right hand side of the global Sobolev inequality can be controlled by the energies $\mathscr{E}_m$ of higher order derivatives, and immediately we have
\begin{multline}\label{eq:entoinfty}
	m^2 \norm[L^\infty(\Sigma_\tau)]{t^d \phi^2} + \norm[L^\infty(\Sigma_\tau)]{\tau^2 t^{d-2} (\partial_t\phi)^2}  + \sum_{i = 1}^d \norm[L^\infty(\Sigma_\tau)]{t^{d-2} (L^i\phi)^2} \\
	\leq C \sum_{k= 0}^{\lfloor \frac{d}2\rfloor + 1} \sum_{i_1, \ldots ,i_k = 1}^d \mathscr{E}_m(L^{i_1}\cdots L^{i_k} \phi,\tau). 
\end{multline}
Combining Lemma \ref{lem:energid} and \eqref{eq:entoinfty} we have the following decay estimate for spatially localized initial data.
\begin{prop}\label{prop:decayestimates}
	Fix $m_0 > 0$. There exists a constant $C$ depending only on the dimension $d$ and $m_0$ such that
	whenever $\phi$ solves \eqref{eq:kg-m} for $m \in (0,m_0]$ with initial data \eqref{eq:kg-id} satisfying $f,g\in C^\infty_0(B(0,1))$, the following decay estimates hold on the region $\{t \geq 2\}$
	\[ m^2 t^{d} |\phi|^2 + t^{d-1} |\partial_t\phi|^2 + t^{d-1} \abs{\nabla \phi}^2 \lesssim \norm[H^{\lfloor \frac{d}2\rfloor+2}]{f}^2 + \norm[H^{\lfloor \frac{d}2\rfloor + 1}]{g}^2.\]
\end{prop}
\begin{proof}
	To control the right hand side of Proposition \ref{prop:globsob} we need to consider the equation satisfied by $L^{i_1} \cdots L^{i_k} \phi$.
	For convenience in the course of the argument we will use $A \lesssim B$ to denote that $A$ is bounded by $B$ up to a universal factor that depends only on $m_0$ and $d$. 
	The Lorentz boosts $L^i$ are well-known to commute with the operator $\Box_m$, hence we can easily check that if $\phi$ solves \eqref{eq:kg-m} with initial data \eqref{eq:kg-id}, then $L^i\phi$ also solves \eqref{eq:kg-m} this time with initial data
	\begin{align*}
		L^i\phi(2,x) &= f'(x) \eqdef 2 \partial_{x^i} f(x) + x^i g(x), \\
		\partial_t L^i\phi(2,x) &= g'(x) \eqdef \partial_{x^i} f(x) + 2 \partial_{x^i} g(x) + x^i \triangle f(x) - x^i m^2 f(x). 
	\end{align*}
	In particular, since $f,g\in C^\infty_0(B(0,1))$, and $m \leq m_0$, 
	\begin{align*}
		 \norm[L^2]{\nabla f'} & \lesssim \norm[L^2]{\nabla^2 f} + \norm[H^1]{g}, \\
		 \norm[L^2]{f'} & \lesssim \norm[L^2]{\nabla f} + \norm[L^2]{g}, \\
		 \norm[L^2]{g'} & \lesssim \norm[H^2]{f} + \norm[L^2]{\nabla g}.
	\end{align*}
	Hence for $k \leq \lfloor \frac{d}{2} \rfloor + 1$, correspondingly the conserved energy 
	\[ \mathscr{E}_m(L^{i_1}\cdots L^{i_k} \phi,\tau) \lesssim \norm[H^{k+1}]{f}^2 + \norm[H^k]{g}^2.\]
	And by \eqref{eq:entoinfty} we obtain
	\begin{multline*}
	m^2 \norm[L^\infty(\Sigma_\tau)]{t^d \phi^2} + \norm[L^\infty(\Sigma_\tau)]{\tau^2 t^{d-2} (\partial_t\phi)^2}  + \sum_{i = 1}^d \norm[L^\infty(\Sigma_\tau)]{t^{d-2} (L^i\phi)^2} \\
		\lesssim \norm[H^{\lfloor \frac{d}{2}\rfloor + 2}]{f}^2 + \norm[H^{\lfloor\frac{d}{2}\rfloor + 1}]{g}^2. 
	\end{multline*}
	
	Next, denote by $D = \{ t \geq 2, |x| \leq t - 1\}$. By finite speed of propagation, the support of the solution $\phi$ when $t \geq 2$ is contained in $D$. We note that $D$ is covered by the $\cup_{\tau \geq \sqrt{3}} \Sigma_\tau$.
	Furthermore, on $D$ we have that 
	\[ \tau^2 = (t+|x|)(t-|x|) \geq t+|x| \geq t.\]
	Additionally, we have
	\[ t\partial_{x^i} = L^i - x^i \partial_t \]
	and $x^i \leq t$ on $D$, meaning that 
	\[ t^{d-1} |\partial_{x^i}\phi|^2 \leq 2 t^{d-3} |L^i \phi|^2 + t^{d-1} |\partial_t \phi|^2 \]
	and proving the theorem. 
\end{proof}

\section{Dispersive estimate}

For $k \geq 0$ denote by $P_k$ the Littlewood-Paley projector on $\Real^d$ to frequency $\approx 2^{k}$. By $P_{-1}$ we denote the low-frequency projector to frequencies $\lessapprox 1$, such that $u = \sum_{j = -1}^\infty P_j u$.

Let $B_i$ be an enumeration of the balls of radius 1 centered at the lattice points $(\frac{1}{\sqrt{d}} \Integer)^d$. The family $\{B_i\}$ is an open cover of $\Real^d$, and every point is contained in no more than $(16 d)^{d/2}$ balls. 
Define $\chi_i$ to be a partition of unity subordinate to $\{B_i\}$; we can require the first $d + 2$ derivatives of the family $\chi_i$ to be uniformly bounded. 
Under these conditions, we have that, for any function $u \in W^{k,1}(\Real^d)$ with $k \leq d+2$, that
\begin{equation}\label{eq:l1comp}
	\norm[W^{k,1}]{u} \leq \sum_i \norm[W^{k,1}]{\chi_i u} \lesssim \sum_i \norm[W^{k,1}(B_i)]{u} \lesssim \norm[W^{k,1}]{u}
\end{equation}
with the implicit constants depending on the dimension $d$ and the uniform bound on the family $\chi_i$. 

Consider $m_0$ a constant fixed once and for all for the remainder of this section. We will now describe the dispersive estimates satisfied by solutions to $\Box_{m_0} \phi = 0$ with initial condition \eqref{eq:kg-id}. 
To do so, we will first decompose the initial data as
\[ f = \sum_{j = -1}^\infty P_j f, \quad g = \sum_{j = -1}^\infty P_j g \]
and observe that we can write 
\[ \phi = \sum_{j = -1}^\infty \phi_j \]
and $\phi_j$ solves $\Box_{m_0} \phi_j = 0$ with initial data $P_jf, P_j g$ prescribed on $\{t = 2\}$. 

For convenience we will write
\[ s_d = \lfloor \frac{d}{2} \rfloor + 1.\]

\subsection{Low frequency estimate for $\phi_{-1}$}

Let $\psi_k$ solve $\Box_{m_0} \psi_k = 0$ with initial data $(\chi_k P_{-1} f, \chi_k P_{-1} g)$.
By linearity we have $\phi_{-1} = \sum_k \psi_k$. 
Proposition \eqref{prop:decayestimates} guarantees that 
\[ (m_0)^2 t^d |\psi_k|^2 + t^{d-1} |\partial \psi_k|^2 \lesssim \norm[H^{s_d + 1}]{\chi_k P_{-1} f}^2 + \norm[H^{s_d}]{\chi_k P_{-1}g}^2. \]
Sobolev embedding $W^{\frac{d}{2},1} \subset L^2$ implies
\[ (m_0)^2 t^d |\psi_k|^2 + t^{d-1} |\partial \psi_k|^2 \lesssim \norm[W^{s_d + \frac{d}{2} + 1,1}]{\chi_k P_{-1} f}^2 + \norm[W^{s_d+ \frac{d}2,1}]{\chi_k P_{-1}g}^2. \]
Taking square roots and summing over $k$ and using the estimate \eqref{eq:l1comp} we get the point-wise estimate
\[
	m_0 t^{d/2} \abs{\phi_{-1}} + t^{(d-1)/2} \abs{\partial\phi_{-1}} \lesssim \norm[W^{s_d + \frac{d}{2} + 1,1}]{P_{-1} f} + \norm[W^{s_d + \frac{d}{2},1}]{P_{-1 g}}
\]
and finally using the fact we have the frequency bound due to the projector $P_{-1}$, we conclude
\begin{equation}
	m_0 t^{d/2} \abs{\phi_{-1}} + t^{(d-1)/2} \abs{\partial\phi_{-1}} \lesssim \norm[L^1]{P_{-1} f} + \norm[L^1]{P_{-1} g}. 
\end{equation}

\subsection{High frequency estimates for $\phi_{k}$, $k \geq 0$}

Given that $\phi_k$ solves $\Box_{m_0} \phi_k = 0$, with initial data $(P_k f, P_k g)$ prescribed at $t = 2$, let us consider the function
\begin{equation}
	\tilde{\phi}_k(t,x) \eqdef \phi_k( 2 + \frac{t-2}{2^k}, \frac{x}{2^k} ). 
\end{equation}
Observe that $\tilde{\phi}_k$ solves $\Box_{2^{-k} m_0} \tilde{\phi}_k = 0$, with initial data $(\tilde{f}_k, \tilde{g}_k)$ prescribed at $t = 2$ given by 
\[ \tilde{f}_k(x) = (P_k f)(2^{-k} x) , \quad \tilde{g}_k(x) = 2^{-k} (P_k g)(2^{-k},x). \]
In particular $\tilde{f}_k(x)$ and $\tilde{g}_k(x)$ have frequency support $\approx 1$. 

Let $\psi_{k,j}$ be the solution to $\Box_{2^{-k} m_0} \psi_{k,j} = 0$, with initial data $(\chi_j \tilde{f}_k, \chi_j \tilde{g}_k)$. Since $k \geq 0$ we can apply Proposition \ref{prop:decayestimates} to obtain estimates on $\psi_{k,j}$, noting that the mass parameter is now $2^{-k}m_0 \leq m_0$. Then summing in analogous fashion to the previous subsubsection, we arrive at 
\begin{equation}
	2^{-k} m_0 t^{d/2} |\tilde{\phi}_k| + t^{(d-1)/2} |\partial \tilde{\phi}_k| \lesssim \norm[L^1]{\tilde{f}_k} + \norm[L^1]{\tilde{g}_k}. 
\end{equation}
We can expand the left hand side and perform a change of variable $s = 2 + \frac{t-2}{2^k}$ to obtain
\begin{multline*}
	2^{-k} m_0 2^{kd/2} (s-2)^{d/2} |\phi_k(s,y)| + 2^{-k} 2^{k(d-1)/2} (s-2)^{(d-1)/2} |\partial\phi_k(s,y)|\\
	\lesssim \norm[L^1]{\tilde{f}_k} + \norm[L^1]{\tilde{g}_k}. 
\end{multline*}
Finally reversing the change of variables for $\tilde{f}_k$ and $\tilde{g}_k$ we obtain that, for $t > 2$, that
\begin{align}
	m_0(t-2)^{d/2} |\phi_k| &\lesssim 2^{kd/2 + k} \norm[L^1]{P_k f} + 2^{kd/2} \norm[L^1]{P_k g}\\
	(t-2)^{(d-1)/2} |\partial \phi_k| & \lesssim 2^{k(d-1)/2 + 2k} \norm[L^1]{P_k f} + 2^{k(d-1)/2 + k} \norm[L^1]{P_k g}.
\end{align}

\section{Discussion}

First, gathering up the estimates in the previous section we obtain the following dispersive estimate. Here we place the initial data at $t = 0$ to make it more comparable to the standard presentations. 

\begin{thm}[Klein-Gordon dispersive estimate]
	Let $m_0 \in [0,M]$ be fixed, and let $d\in \Integer_+$. Consider the initial value problem 
	\begin{equation}\label{eq:ivpstmt}
		\begin{cases}
			\Box_{m_0} \phi = 0 \\
			\phi(0,x) = f(x) \\
			\partial_t\phi(0,x) = g(x)
		\end{cases}
	\end{equation}
	Then there exists a universal constant $C$ depending only on $M$ and $d$, such that the following dispersive estimates hold:
	\begin{subequations}
		\begin{align}
			m_0 |P_{-1} \phi| & \leq \frac{C}{(1+t)^{d/2}} (\norm[L^1]{P_{-1} f} + \norm[L^1]{P_{-1} g}) \label{eq:lowfreq}\\
			m_0 |P_{k}\phi|_{k \geq 0} & \leq \frac{C\cdot 2^{kd/2}}{t^{d/2}} ( 2^k \norm[L^1]{P_{k} f} + \norm[L^1]{P_k g} ) \label{eq:highfreq}\\
			|\partial P_k\phi|_{k \geq 0} & \leq \frac{C \cdot 2^{k(d-1)/2}}{t^{(d-1)/2}} (2^{2k} \norm[L^1]{P_k f} + 2^k \norm[L^1]{P_k g}). \label{eq:wavedecay}
		\end{align}
	\end{subequations}
\end{thm}

\begin{rmk}
	\begin{enumerate}
		\item The theorem makes clear the heuristic that ``high frequency parts behave like solutions to the wave equation''. Indeed, the decay estimate \eqref{eq:wavedecay} is the only one that will survive in the vanishing mass limit $m_0 \to 0$. 
		\item The latter two estimates can be combined. By replacing $P_k f \mapsto \partial \triangle^{-1} P_k f$, we see that \eqref{eq:wavedecay} implies 
			\[ |P_k \phi| \leq \frac{C \cdot 2^{k(d-1)/2}}{t^{(d-1)/2}} (2^k \norm[L^1]{P_k f} + \norm[L^1]{P_k g}).\]
			Indeed, interpolating between the two we find in fact that, for any $s \in [\frac{d-1}{2}, \frac{d}2]$, 
			\[ |P_k \phi| \leq \frac{C \cdot 2^{ks}}{t^{s}} (2^k \norm[L^1]{P_k f} + \norm[L^1]{P_k g}), \]
			showing the trade off between regularity and decay for the Klein-Gordon equation. 
	\end{enumerate}
\end{rmk}

\printbibliography
\end{document}